\documentclass{amsart}
\usepackage[utf8]{inputenc}

\usepackage{tikz}
\usepackage{wrapfig}
\usetikzlibrary{arrows}
\usepackage{float}
\usepackage{subcaption} 
\usepackage{amsmath}
\usepackage{forest}
\usepackage[font=footnotesize,labelfont=bf]{caption}
\usepackage{latexsym,array,delarray,amsthm,amssymb,epsfig,setspace, txfonts}
\usepackage{mathtools}
\usepackage{xcolor}

\theoremstyle{plain}
\newtheorem{thm}{Theorem}[section]

\newtheorem*{thm*}{Theorem}
\newtheorem*{lemma*}{Lemma}
\newtheorem*{prop*}{Proposition}
\newtheorem*{cor*}{Corollary}
\newtheorem*{conj*}{Conjecture}
\newtheorem*{alg*}{Algorithm}

\theoremstyle{definition}

\newtheorem{ex}[thm]{Example}

\theoremstyle{remark}
\newtheorem*{rmk}{Remark}

\usepackage[linesnumbered, algoruled, boxed, lined]{algorithm2e}
\SetKwInOut{Input}{Input}
\SetKwInOut{Output}{Output}
\SetKwIF{If}{ElseIf}{Else}{if}{then}{else if}{else}{}
\SetKwFor{For}{for}{do}{}

\newcommand{\rr}{\mathbb{R}}

\newcommand{\bfk}{\mathbf{k}}

\newcommand{\cg}{\mathcal{G}}
\newcommand{\cd}{\mathcal{D}}

\newcommand{\ind}{\mbox{$\perp \kern-5.5pt \perp$}}

\DeclareMathOperator{\image}{image}
\newcommand{\pa}{\mathrm{pa}}
\newcommand{\an}{\mathrm{an}}

\title{Discrete Max-Linear Bayesian Networks}
\author{Benjamin Hollering and Seth Sullivant}

\begin{document}

\maketitle

\begin{abstract}
    Discrete max-linear Bayesian networks are directed graphical models specified by the same recursive structural equations as max-linear models but with discrete innovations. When all of the random variables in the model are binary, these models are isomorphic to the conjunctive Bayesian network (CBN) models of Beerenwinkel, Eriksson, and Sturmfels. Many of the techniques used to study CBN models can be extended to discrete max-linear models and similar results can be obtained. In particular, we extend the fact that CBN models are toric varieties after linear change of coordinates to all discrete max-linear models. 
\end{abstract}

\section{Introduction}
Max-linear Bayesian networks are a special class of graphical models introduced in \cite{GK18} to model extreme events. A max-linear Bayesian network, $X = (X_1, \ldots X_n)$ is determined by a directed acyclic graph $\cd = (V, E)$, edge weights $c_{ij} \geq 0$, and independent positive random variables $Z_1 \ldots Z_n$ called \emph{innovations}. The $Z_i$ have support $(0, \infty)$ and are required to have atom-free distributions. Then the random vector $X$ is required to satisfy the structural equations
\[
X_i = \bigvee_{j \in \pa(i)} c_{ij}X_j  \vee Z_i,
\]
where $\vee$ denotes maximum and $\pa(i)$ denotes the parents of vertex $i$ in $\cd$. These models are able to more accurately model extreme events spreading throughout a network than standard discrete or Gaussian Bayesian networks \cite{GK18, KL19} and also have interesting ties to tropical geometry \cite{AKLT20}. 

Max-linear models also exhibit a remarkable property concerning conditional independence: they are usually not \emph{faithful} to their underlying DAG, meaning that they often satisfy more conditional independence statements than those implied by the d-separation criterion \cite{KL19}. Am\'{e}ndola, Kl\"uppelberg, Lauritzen, and Tran recently gave a new criterion named \emph{$\ast$-separation} which gives a complete set of conditional independence statements for max-linear models \cite{AKLT20}. There has also been work done to establish when the parameters of the model are identifiable \cite{GK18, GKS20, KL19}. In all of this work, the atom-free property of the innovations is key. 

In this note, we consider discrete max-linear Bayesian networks so we assume the $Z_i$ are all $k$-state discrete random variables. This means the $Z_i$ are now atomic so much of the above work on max-linear models does not apply; however, these models are closely related to the conjunctive Bayesian network (CBN) models discussed in \cite{BES06, BES07}. CBN models were originally introduced in \cite{BES06} to model HIV drug resistance and cancer development via mutation in a genome. In CBN models, events accumulate as you move up a poset (or DAG) similar to the way extreme events spread throughout a max-linear Bayesian network. Part of our goal in this note is to show that CBN models are naturally included in the family of max-linear models. Though we do not develop the details here, we also believe it would be interesting to find a relationship between the typical max-linear model and our discrete analogue similar to the relationship found between the CBN model and its continuous analogue as the authors did in \cite{BS09}.

In Section 2 we give a brief overview of the combinatorial objects we use throughout the note and some background on the CBN model. In Section 3 we define the discrete max-linear Bayesian network model and show that when all of the random variables involved are binary, it is equal to the CBN model after a natural change of coordinates. In Section 4 we describe the algebraic structure of discrete max-linear model.

\section{Preliminaries}
In this section we provide some basic background on graphs and posets which we will use throughout this note. We also discuss Conjunctive Bayesian Network (CBN) models which were  first introduced in \cite{BES06}. These models are directly related to the discrete max-linear models we discuss in this note. 

\subsection{Graphs, Posets, and Lattices}
Our notation for graphs follows that of \cite{AKLT20} while our notation for posets and lattices follows that of \cite{RS12}. The graphs in this note will primarily be directed graphs which are given by a vertex set $V = \{1, \ldots n\}$ and an edge set $E = \{j \to i : i, j \in V, ~i \neq j \}$. A vertex $j$ is a \emph{parent} of $i$ if $j \to i \in E$  and we denote the set of parents of $i$ by $\pa(i)$. A \emph{path} from $j$ to $i$ is a sequence of distinct nodes $[j = \ell_0, \ell_1, \ldots, \ell_m = i]$ such that either $\ell_r \to \ell_{r+1}$ or $\ell_{r+1} \to \ell_{r}$ for each $r$. A \emph{directed path} from $j$ to $i$ is a path where $\ell_r \to \ell_{r+1}$ for all $r$ and a \emph{directed cycle} is a directed path where $j = i$. A vertex $j$ is an \emph{ancestor} of $i$ if there exists a directed path from $j$ to $i$ and we denote the set of ancestors of $i$ by $\an(i)$.  Most directed graphs in this note will be \emph{directed acyclic graphs} (DAG) which are directed graphs with no directed cycles. 

We will also frequently use partially ordered sets, typically called posets, throughout this note. A \emph{poset}, $P$, is a set equipped with a binary relation $\leq$ that is reflexive, antisymmetric, and transitive. A chain is a poset where any two elements are comparable and we denote a chain whose elements are $[k] = \{1, \ldots, k\}$ with the natural relation by $\bfk$. An element $j$ of $P$ is said to \emph{cover} $i$,
if $j > i$ and there is no other element $k$ such that $j> k > i$.
A poset can be represented by an undirected graph called the \emph{Hasse diagram}
of the poset where the vertices correspond to elements of $P$, edges
represent cover relations, and the elements are ordered in the figure
so smaller elements are at the bottom.  See Figure \ref{fig:CBNexample}
for an example.

An \emph{order ideal} of $P$ is a subset $I$ of $P$ such that if $i \in I$ and $j \leq i$ then $j \in I$. A \emph{dual order ideal} of $P$ is a subset $I$ of $P$ such that if $i \in P$ and $j \geq i$ then $j \in P$. An \emph{order-preserving map} is a map $\phi: P \to Q$ between posets such that if $j, i \in P$ satisfy $j \leq i$ then $\phi(j) \leq \phi(i)$. 

Many of the posets we work with will be obtained by taking the \emph{transitive closure} of a DAG. Let $\cd$ be a DAG and define a relation on the vertices by $j \leq i$ if and only if there is a directed path from $j$ to $i$. We refer to this poset as $\cd^\mathrm{tr}$, which denotes the transitive closure poset.

A lattice, $L$, is a poset equipped with two binary operations, denoted $\vee$ and $\wedge$,  such that for any elements $s, t \in L$ the least upper bound of $s$ and $t$ exists which is $s \vee t$ and the greatest lower bound of $s$ and $t$ exists which is $s \wedge t$. The operation $\vee$ is typically referred to as the \emph{join} while $\wedge$ is called the \emph{meet}. A classic example that will arise in this note is the lattice of order ideals of a finite poset. Let $P$ be a finite poset and denote by $J(P)$ the order ideals of $P$ ordered by inclusion then $J(P)$ is a lattice with meet given by intersection and join given by union. Furthermore, $J(P)$ is a \emph{distributive lattice} meaning the operations $\vee$ and $\wedge$ distribute over each other. Throughout this note we will also use $\vee$ to denote the maximum of a subset of a totally ordered set and $\wedge$ to denote the minimum since these operations are the join and meet respectively for finite subsets of totally ordered sets. 


\subsection{Conjunctive Bayesian Networks}
\label{sec:CBN}
In this section we provide some basic background on the CBN model which was first introduced by Beerenwinkel, Eriksson, and Sturmfels in \cite{BES06} as a mathematical model for mutations occurring in a genome. For further information we refer the reader to \cite{BES06, BES07}. 
In Section \ref{sec:DMLBN}, we show that CBN models can be viewed as a subclass of D-MLBN models which are our main focus in this note. 

The CBN model begins with a poset $P$ whose elements are called \emph{events} which are usually taken to be the elements of $[n]$. The state space of the model is the lattice of order ideals $J(P)$ and elements $g \in J(P)$ are called \emph{genotypes}. It is often convienent to think of a state $g$ both as a subset of the ground set $[n]$ and as a 0-1 string and we will do so throughout the note. 

The parameterization of the CBN model can be explicitly written down in terms of the poset $P$ but it is often convenient to instead view it as a directed graphical model.  We refer the reader to  \cite[Chapter 13]{SS18} for additional information on parameterizations of graphical models. 
We first form a DAG with edges $i \to j$ if $i < j$ is a cover relation in $P$ and associate a binary random variable $X_i$ to each $i \in P$. Then the CBN model is a directed graphical model on the $X_i$ with conditional probabilities given by
\begin{equation}
(P(X_i = b | X_{\pa(i)} = a))_{a \in \{0, 1\}^{\pa(i)}, b \in \{0, 1\}}
=
\begin{pmatrix}
\label{eqn:cbnCondProbs}
1 & 0 \\
\vdots & \vdots \\
1 & 0 \\
\theta_0^{(i)} & \theta_1^{(i)}
\end{pmatrix}
\end{equation}
where the rows and columns are ordered lexicographically and 
$\theta_1^{(i)}$ is simply the conditional probability that the event $i$ occurs given all of its parents have occurred while $\theta_0^{(i)}$ is the conditional probability that the event $i$ does not occur given all of its parents have occurred. This means that $\theta_0^{(i)} + \theta_1^{(i)} = 1$. We note that this slightly different than the matrix of conditional probabilities shown in \cite{BES07} but it is equivalent. The probability of observing an event $g$ is the product of these conditional probabilities so
\[
p_g = P(X = g) = \prod_{i \in P} P(X_i = g_i | X_{\pa(i)} = g_{\pa(i)}).
\]

\begin{ex}
\label{ex:CBN}
Let $P$ be the poset pictured on the left of Figure \ref{fig:CBNexample}. The state space of the CBN model on $P$ is the lattice of order ideals $J(P)$ pictured on the right of Figure \ref{fig:CBNexample}. In this representation, the 0-1 strings represent the corresponding subsets of the ground set of $P$. For instance, the element $(1, 1, 1, 0, 1)$ of $J(P)$ represents the order ideal $\{1, 2, 3, 5\}$ of $P$. The parameterization of the CBN model on $P$ is
\begin{align*}
p_{00000} &= \theta_0^{(1)} \theta_0^{(2)}                                               &
p_{10000} &= \theta_1^{(1)} \theta_0^{(2)}                                              \\
p_{01000} &= \theta_0^{(1)} \theta_1^{(2)}                                               &
p_{11000} &= \theta_1^{(1)} \theta_1^{(2)} \theta_0^{(3)}                               \\
p_{11100} &= \theta_1^{(1)} \theta_1^{(2)} \theta_1^{(3)} \theta_0^{(4)} \theta_0^{(5)}  &
p_{11110} &= \theta_1^{(1)} \theta_1^{(2)} \theta_1^{(3)} \theta_1^{(4)} \theta_0^{(5)}  \\
p_{11101} &= \theta_1^{(1)} \theta_1^{(2)} \theta_1^{(3)} \theta_0^{(4)} \theta_1^{(5)}  &
p_{11111} &= \theta_1^{(1)} \theta_1^{(2)} \theta_1^{(3)} \theta_1^{(4)} \theta_1^{(5)}. 
\end{align*}

Note that we suppress the commas and parentheses when writing $p_g$ which we will do throughout this note. 
\end{ex}

\begin{figure}
    \centering
    \begin{subfigure}[b]{0.3 \linewidth}
        \begin{tikzpicture}[-, >=stealth', shorten >=1pt, auto, node distance=1cm, semithick]
        \tikzstyle{every node}=[circle, line width =1pt, font=\scriptsize, minimum height =0.65cm]
            \node (i1) [draw] {1};
            \node (i2) [right of = i1, xshift = 1cm, draw] {2};
            \node (i3) [above of = i1, xshift = 1.0cm, draw] {3};
            \node (i4) [above of = i3, draw, xshift = -1cm] {4};
            \node (i5) [above of = i3, draw, xshift = 1cm] {5};
            
            \path (i1) edge (i3);
            \path (i2) edge (i3);
            \path (i3) edge (i4);
            \path (i3) edge (i5);
            
            \node (label) [below of = i1, xshift = 1.0cm] {\scalebox{1.5}{$P$}};
        \end{tikzpicture}
    \end{subfigure}
    \begin{subfigure}[b]{0.3 \linewidth}
        \begin{tikzpicture}
            \node (i1) {$(0,0,0,0,0)$};
            \node (i2) [above left of = i1, xshift = -1.0cm, yshift = .3cm] {$(1, 0, 0, 0, 0)$};
            \node (i3) [above right of = i1, xshift =  1.0cm, yshift = .3cm] {$(0, 1, 0, 0, 0)$};
            \node (i4) [above right of = i2, xshift =  1.0cm, yshift = .3cm] {$(1, 1, 0, 0, 0)$};
            \node (i5) [above of = i4] {$(1, 1, 1, 0, 0)$};
            \node (i6) [above left of = i5, xshift =  -1.0cm, yshift = .3cm] {$(1, 1, 1, 1, 0)$};
            \node (i7) [above right of = i5, xshift =  1.0cm, yshift = .3cm] {$(1, 1, 1, 0, 1)$};
            \node (i8) [above right of = i6, xshift =  1.0cm, yshift = .3cm] {$(1, 1, 1, 1, 1)$};
            
            \path (i1) edge (i2);
            \path (i1) edge (i3);
            \path (i2) edge (i4);
            \path (i3) edge (i4);
            \path (i4) edge (i5);
            \path (i5) edge (i6);
            \path (i5) edge (i7);
            \path (i6) edge (i8);
            \path (i7) edge (i8);
            
            \node (label) [below of = i1] {$J(P)$};
        \end{tikzpicture}
    \end{subfigure}
    \caption{The hasse diagram of a poset $P$ and its lattice of order ideals $J(P)$ written as 0-1 strings instead of subsets of $\{1, 2, 3, 4, 5\}$.}
    \label{fig:CBNexample}
\end{figure}
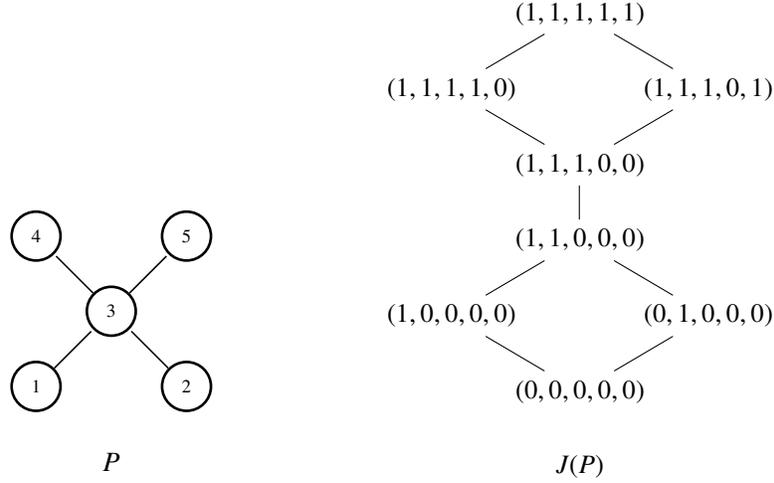

\section{Discrete Max-Linear Bayesian Networks}
\label{sec:DMLBN}
In this section we introduce a discrete version of the max-linear models studied in \cite{AKLT20, GK18, KL19}. When all of the random variables in the model are binary, there is a direct correspondence between discrete max-linear models and the conjunctive Bayesian networks discussed in \cite{BES06, BES07}. For this reason, we try to keep our notation consistent with that of \cite{BES07} when possible. 

Let $\cd = (V, E)$ be a directed acyclic graph (DAG) and associated a discrete random variable $Z_i$ with $k$ states to each vertex $i \in V$. Then the \emph{discrete max-linear Bayesian network} (D-MLBN) is the family of joint distributions of the random variables $(X_i)_{i \in V}$ specified by
\begin{equation}
\label{eqn:recStructEq}
    X_i = \bigvee_{j \in \pa(i)} X_j \vee Z_i, ~~~~~~ i \in V
\end{equation}
where $\pa(i)$ denotes the \emph{parents} of vertex $i$ in $\cd$. These are the same structural equations used to specify the max-linear Bayesian networks discussed in \cite{AKLT20, GK18, KL19} except there are no coefficients and the random variables $Z_1, \ldots Z_n$ are now discrete instead of continuous and atom-free. Despite these alterations, this system of equations still has the same solution which is
\begin{equation}
\label{eqn:solvedStructEq}
    X_i = \bigvee_{j \in \an(i) \cup \{ i \}} Z_j, ~~~~~~ i \in V
\end{equation}
where $\an(i)$ denotes the \emph{ancestors} of vertex $i$ in $\cd$.

The state space of the $k$-state discrete max-linear model is the set of order-preserving maps from the transitive closure of $\cd$ to a chain of length $k$. More explicitly, let $\bfk$ be a chain of size $k$ and let $\cd^\mathrm{tr}$ be the poset obtained by taking the transitive closure of $\cd$. Then $g = (g_1, \ldots g_n)$ is a possible state of the $k$-state D-MLBN if there exists an order-preserving map $\pi: \cd^\mathrm{tr} \to \bfk$ such that $\pi(i) = g_i$. This can be seen by directly examining the structural equations and their solution. Note that if $i \geq j$ in the partial order $D^\mathrm{tr}$, then $j$ is an ancestor of $i$ so there is a directed path $[j = \ell_0, \ell_1, \ldots, \ell_m = i]$ and for any $r$ it is immediate that, $\ell_{r} \in \pa(\ell_{r+1})$ which immediately implies that $X_{l_{r+1}} \geq X_{l_r}$. This gives a chain of inequalities from which we get $X_j \leq X_i$. Denote this set of states by $\cg(\cd, k)$ and note that this forms a distributive lattice with meet given by taking the coordinate-wise minimum and join given by taking the coordinate-wise maximum. When it is clear from context, we simply write $\cg$ instead of $\cg(\cd, k)$. 
This is analogous to the state space of the CBN model being the lattice of order ideals of a poset \cite{BES07}.  

\begin{ex}
\label{ex:dmlbn}
Let $\cd$ be the DAG pictured in Figure \ref{fig:dmlbnExample}. Then the structural equations of the D-MLBN model on $D$ are
\begin{align*}
    X_1 &= Z_1, \ \ X_2 = Z_2, \ \ X_3 = X_1 \vee Z_3 \\ 
    X_4 &= X_1 \vee X_2 \vee Z_4, \ \ X_5 = X_3 \vee X_4 \vee Z_5
\end{align*}
which have the solution
\begin{align*}
    X_1 &= Z_1, \ \ X_2 = Z_2, \ \ X_3 = Z_1 \vee Z_3 \\ 
    X_4 &= Z_1 \vee Z_2 \vee Z_4, \ \ X_5 = Z_1 \vee Z_2 \vee Z_3 \vee Z_4 \vee Z_5.
\end{align*}
If each of the $Z_i$ has two states, so $k = 2$, then the state space of the model is the lattice $\cg$ that is also pictured in Figure \ref{fig:dmlbnExample}. 
\end{ex}

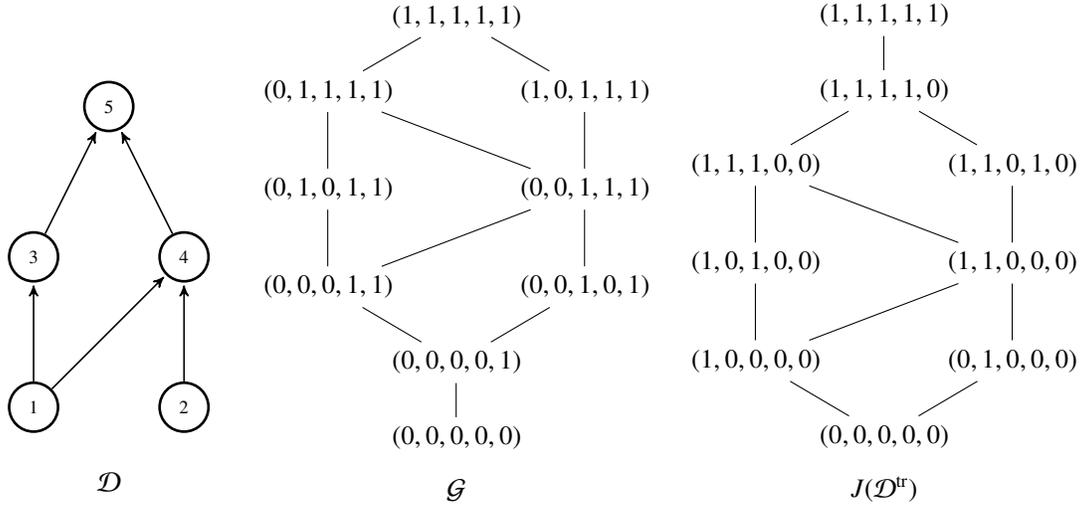
\begin{figure}
    \centering
    \begin{subfigure}[b]{0.2 \linewidth}
        \begin{tikzpicture}[->, >=stealth', shorten >=1pt, auto, node distance=2cm, semithick]
        \tikzstyle{every node}=[circle, line width =1pt, font=\scriptsize, minimum height =0.65cm]
            \node (i1) [draw] {1};
            \node (i2) [right of = i1, draw] {2};
            \node (i3) [above of = i1, draw] {3};
            \node (i4) [above of = i2, draw] {4};
            \node (i5) [above of = i3, xshift = 1.0cm, draw] {5};
            
            \path (i1) edge (i3);
            \path (i1) edge (i4);
            \path (i2) edge (i4);
            \path (i3) edge (i5);
            \path (i4) edge (i5);
            
            \node (label) [below of = i1, xshift = 1.0cm, yshift = 1cm] {\scalebox{1.5}{$\cd$}};
        \end{tikzpicture}
    \end{subfigure}
    \begin{subfigure}[b]{0.35 \linewidth}
        \begin{tikzpicture}
            \node (i0) {$(0, 0, 0, 0, 0)$};
            \node (i1) [above of = i0]{$(0, 0, 0, 0, 1)$};
            \node (i2) [above left of = i1, xshift = -1.0cm, yshift = .3cm] {$(0, 0, 0, 1, 1)$};
            \node (i3) [above right of = i1, xshift = 1.0cm, yshift = .3cm] {$(0, 0, 1, 0, 1)$};
            \node (i4) [above of = i2, yshift = .3cm] {$(0, 1, 0, 1, 1)$};
            \node (i5) [above of = i3, yshift = .3cm] {$(0, 0, 1, 1, 1)$};
            \node (i6) [above of = i4, yshift = .3cm] {$(0, 1, 1, 1, 1)$};
            \node (i7) [above of = i5, yshift = .3cm] {$(1, 0, 1, 1, 1)$};
            \node (i8) [above of = i1, yshift = 3.6cm] {$(1, 1, 1, 1, 1)$};
            
            \path (i0) edge (i1);
            \path (i1) edge (i2);
            \path (i1) edge (i3);
            \path (i2) edge (i4);
            \path (i2) edge (i5);
            \path (i3) edge (i5);
            \path (i4) edge (i6);
            \path (i5) edge (i6);
            \path (i5) edge (i7); 
            \path (i6) edge (i8);
            \path (i7) edge (i8);
            
            \node (label) [below of = i0, yshift = .3cm] {\scalebox{1}{$\cg$}};
        \end{tikzpicture}
    \end{subfigure}
    \begin{subfigure}[b]{0.3 \linewidth}
        \begin{tikzpicture}
            \node (i1) [above of = i0, yshift = -1cm]{$(0, 0, 0, 0, 0)$};
            \node (i2) [above left of = i1, xshift = -1.0cm, yshift = .3cm] {$(1, 0, 0, 0, 0)$};
            \node (i3) [above right of = i1, xshift = 1.0cm, yshift = .3cm] {$(0, 1, 0, 0, 0)$};
            \node (i4) [above of = i2, yshift = .3cm] {$(1, 0, 1, 0, 0)$};
            \node (i5) [above of = i3, yshift = .3cm] {$(1, 1, 0, 0, 0)$};
            \node (i6) [above of = i4, yshift = .3cm] {$(1, 1, 1, 0, 0)$};
            \node (i7) [above of = i5, yshift = .3cm] {$(1, 1, 0, 1, 0)$};
            \node (i8) [above of = i1, yshift = 3.6cm] {$(1, 1, 1, 1, 0)$};
            \node (i9) [above of = i8] {$(1, 1, 1, 1, 1)$};
            
            \path (i1) edge (i2);
            \path (i1) edge (i3);
            \path (i2) edge (i4);
            \path (i2) edge (i5);
            \path (i3) edge (i5);
            \path (i4) edge (i6);
            \path (i5) edge (i6);
            \path (i5) edge (i7); 
            \path (i6) edge (i8);
            \path (i7) edge (i8);
            \path (i8) edge (i9);
            
            \node (label) [below of = i0, yshift = .3cm] {\scalebox{1}{$J(\cd^\mathrm{tr})$}};
        \end{tikzpicture}
    \end{subfigure}
    \caption{A DAG $\cd$ and the state space, $\cg$, of the 2-state D-MLBN model pictured as the lattice of order preserving maps from $\cd$ to the chain $\mathbf{2}$. Also pictured is the lattice of order ideals $J(\cd^\mathrm{tr})$ of the poset $\cd^\mathrm{tr}$ written as 0-1 vectors.} 
    \label{fig:dmlbnExample}
\end{figure}

Similarly to the CBN model, the D-MLBN model can also be thought of as a directed graphical model on a DAG $\cd$ with additional restrictions on the parameters. In the usual directed graphical model, the parameters of the model are the conditional probabilities $P(X_i = x_i | X_{\pa(i)} = x_{\pa(i)})$ \cite{GSS05}. In the D-MLBN model, we can compute these conditional probabilities in terms of the distributions of the $Z_i$. Let each $Z_i$ have distribution 
$\theta^{(i)} = (\theta_0^{(i)}, \ldots, \theta_{k-1}^{(i)}) \in \Delta_{k-1}$ and $g \in \{0, \ldots k-1\}^n$. For any $i$ define $M_i = \bigvee_{j \in \pa(i)} g_j$ then we have
\begin{equation}
\label{eqn:dmlbnCondProbs}
   P(X_i = g_i | X_{\pa(i)} = g_{\pa(i)}) =
\begin{cases}
0, & g_i < M_i \\
\sum_{\ell \leq M_i} \theta_\ell^{(i)}, & g_i = M_i\\
\theta_{g_i}^{(i)}, & g_i > M_i
\end{cases}.  
\end{equation}
Using these conditional probabilities, we can then compute the probability $p_g = P(X = g)$ which is
\[
p_g = \prod_{i \in V} P(X_i = g_i | X_{\pa(i)} = g_{\pa(i)}).
\]
Note that if $g$ does not come from an order preserving map, that is $g \notin \cg(\cd, k)$, then $p_g = 0$. This means we can either think of the D-MLBN model as just being a model for the states $\cg(\cd, k)$ or a model for all $g \in \{0, \ldots k-1\}^n$ where the $g \notin \cg(\cd, k)$ have probability 0. Both of these perspectives can be useful when studying the algebraic structure of the model. The following example illustrates this parametric description of the model. 

\begin{ex}
We again let $\cd$ and $\cg$ be the DAG and lattice pictured in Figure \ref{fig:dmlbnExample}. Then from the above discussion the parameterization of the D-MLBN model is
\begin{align*}
p_{00000} &= \theta_0^{(1)} \theta_0^{(2)} \theta_0^{(3)} \theta_0^{(4)} \theta_0^{(5)}  &
p_{00001} &= \theta_0^{(1)} \theta_0^{(2)} \theta_0^{(3)} \theta_0^{(4)} \theta_1^{(5)} \\
p_{00011} &= \theta_0^{(1)} \theta_0^{(2)} \theta_0^{(3)} \theta_1^{(4)}                 &
p_{00101} &= \theta_0^{(1)} \theta_0^{(2)} \theta_1^{(3)} \theta_0^{(4)}                \\
p_{00111} &= \theta_0^{(1)} \theta_0^{(2)} \theta_1^{(3)} \theta_1^{(4)}                 &
p_{01011} &= \theta_0^{(1)} \theta_1^{(2)} \theta_0^{(3)}                               \\
p_{01111} &= \theta_0^{(1)} \theta_1^{(2)} \theta_1^{(3)}                                &
p_{10111} &= \theta_1^{(1)} \theta_0^{(2)}                                              \\
p_{11111} &= \theta_1^{(1)} \theta_1^{(2)}. 
\end{align*}
Note that while this appears to be a monomial map, the relationship $\theta_0^{(i)} + \theta_1^{(i)} = 1$ implies that it is not and so the ideal in these coordinates is not toric.  
\end{ex}

We now describe the relationship between the 2-state D-MLBN model and the CBN model of \cite{GSS05}. 

\begin{thm}
\label{thm:CBNisoDMLBN}
Let $\cd = (V, E)$ be a DAG and $\rho$ be the map parameterizing the CBN model on 
$\cd^\mathrm{tr}$. Let $\psi$ be the map parameterizing the 2-state D-MLBN model on $\cd$. Then $\image(\rho)$ is equal to $\image(\psi)$ after a natural relabeling of coordinates. In particular, there exists a bijection, 
$\phi: J(\cd^\mathrm{tr}) \to \cg(\cd, 2)$ such that
$\rho_g(\theta_0^{(1)}, \theta_1^{(1)}, \ldots, \theta_0^{(n)}, \theta_1^{(n)}) = \psi_{\phi(g)}(\theta_1^{(1)}, \theta_0^{(1)}, \ldots, \theta_1^{(n)}, \theta_0^{(n)})$.
\end{thm}
\begin{proof}
We first describe the bijection between the state spaces of the two models.  Recall that the state space of the CBN model is the distributive lattice $J(\cd^\mathrm{tr})$ of order ideals.
There is a natural bijection between elements $g \in J(\cd^\mathrm{tr})$ and order-preserving maps from $\pi : \cd^\mathrm{tr} \to \{0, 1\}$. The general form of the following bijection can be found in \cite[Proposition 3.5.1]{RS12} but we describe the special case for order-preserving maps to $\{0, 1\}$ here. For any order ideal $g$ of $\cd^\mathrm{tr}$ let $\pi_g$ be the map defined by
\[
\pi_g(i) = 
\begin{cases}
0, &  i \in g \\
1, &  i \notin g, 
\end{cases}
\]
It remains to show that $\rho_g(\theta_0^{(1)}, \theta_1^{(1)}, \ldots, \theta_0^{(n)}, \theta_1^{(n)}) = \psi_{\phi(g)}(\theta_1^{(1)}, \theta_0^{(1)}, \ldots, \theta_1^{(n)}, \theta_0^{(n)})$. We do this by showing that the two models have the same conditional probabilities after interchanging $\theta_0^{(i)}$ and $\theta_1^{(i)}$. 

Let $\phi(g) \in \cg(\cd, 2)$ be a state of the D-MLBN model and for each $i$ let $M_i = \bigvee_{j \in \pa(i)} \phi(g)_j$. Observe that
\[
M_i = \bigvee_{j \in \pa(i)} \phi(g)_j =
\begin{cases}
0, ~~\phi(g)_j = 0 \mbox{ for all } j \in \pa(i)   \\
1, \mbox{ otherwise }
\end{cases}
= 
\begin{cases}
0, ~~g_j = 1 \mbox{ for all } j \in \pa(i)  \\
1, \mbox{ otherwise }
\end{cases}
\]
with the second equality following from the definition of $\phi$. We now examine the different possibilities for $\phi(g)_i$ and $M_i$ and compute the conditional probabilities in each case. 

Suppose $\phi(g)_i = 0$ and $M_i = 0$, then under the D-MLBN model, we have
\[
P(X_i = \phi(g)_i | X_{\pa(i)} = \phi(g)_{\pa(i)}) = \sum_{\ell \leq 0} \theta_\ell^{(i)} = \theta_0^{(i)}.
\]
We know from the above formula for $M_i$ that if $M_i = 0$, then for all $j \in \pa(i)$, $g_j = 1$. Since $\phi(g)_i = 0$ we have that $g_i = 1$ and so the corresponding entry of the matrix of conditional probabilities for the CBN model in Equation \ref{eqn:cbnCondProbs} is $\theta_1^{(i)}$. 

If $\phi(g)_i = 1$ and $M_i = 0$, then under the D-MLBN model, we have
\[
P(X_i = \phi(g)_i | X_{\pa(i)} = \phi(g)_{\pa(i)}) =  \theta_1^{(i)}. 
\]
On the other hand, we now have $g_i = 0$ and $M_i = 0$ so the conditional probability for the CBN model is $\theta_0^{(i)}$. 

It is straightforward to check the remaining cases so we omit it here. In these cases both models have the same conditional probabilities which are either 0 or 1. Since each model is a directed graphical model, the probability of observing $g$ (or $\phi(g)$) is simply
\begin{align*}
    \rho_g(\theta) &= \prod_{i \in V} P_\mathrm{CBN}(X_i = g_i | X_{\pa(i)} = g_{\pa(i)}) \\
    \psi_g(\theta) &= \prod_{i \in V} P_{\mathrm{D-MLBN}}(X_i = g_i | X_{\pa(i)} = g_{\pa(i)}). \\
\end{align*}
Since we know that the conditional probabilities are equal after interchanging the parameters $\theta_0^{(i)}$ and $\theta_1^{(i)}$ for each $i$, we have that the above products are equal after interchanging the corresponding parameters as claimed above. 
\end{proof}

We end this section with an example that illustrates the previous theorem. 

\begin{ex}
\label{ex:CBNvsDMLBN}
Let $\cd$ be the DAG pictured on the left of Figure \ref{fig:dmlbnExample}. We have already seen that $\cg = \cg(\cd, 2)$ pictured in the middle of Figure \ref{fig:dmlbnExample} is the state space of the D-MLBN model on $\cd$ while the lattice of order ideals $J(\cd^\mathrm{tr})$ that is pictured on the right of Figure \ref{fig:dmlbnExample} is the state space of the CBN model on $\cd^\mathrm{tr}$. 

The map $\phi$ from Theorem \ref{thm:CBNisoDMLBN} maps the element $g = (1, 0, 1, 0, 0) \in J(\cd^\mathrm{tr})$ to the element $\phi(g) = (0, 1, 0, 1, 1) \in \cg$. The probability of $g$ under the CBN model is
\[
\rho_{10100}(\theta_0^{(1)}, \theta_1^{(1)}, \ldots, \theta_0^{(5)}, \theta_1^{(5)}) = \theta_1^{(1)} \theta_0^{(2)} \theta_1^{(3)}  
\]
while the probability of $\phi(g)$ under the D-MLBN model is
\[
\psi_{01011}(\theta_0^{(1)}, \theta_1^{(1)}, \ldots, \theta_0^{(5)}, \theta_1^{(5)}) = \theta_0^{(1)} \theta_1^{(2)} \theta_0^{(3)}.
\]
We can see that these two probabilities will be equal after interchanging the parameters $\theta_0^{(i)}$ and $\theta_1^{(i)}$.

\end{ex}

\section{Algebraic Structure of the D-MLBN Model}
In this section we describe the algebraic structure of the D-MLBN model. We do this by extending the techniques developed for CBN models in \cite{BES07} to all D-MLBN models. The main tool here is M\"obius inversion which corresponds to a linear change of coordinates on the D-MLBN model. In these new coordinates, the ideal of polynomials that vanish on a D-MLBN belongs to a special class of \emph{toric ideals} whose Gr\"obner bases were described by Hibi in \cite{TH87}. 

Let $\cd$ be a DAG and $\psi_\cd^{(k)}$ be the map parameterizing the $k$-state D-MLBN model on $\cd$ and $\cg$ be the state space of the model. Also let $\rr[p_g] = \rr[p_g : g \in \cg]$, then our goal is to find a Gr\"obner basis for the ideal
\[
I_\cd^{(k)} = \{ f \in \rr[p_g] : f(a) = 0 \mbox{ for all } a \in \image(\psi_\cd^{(k)})\}. 
\]
Finding a Gr\"obner basis for the ideal $I_\cd$ is a first step in obtaining an \emph{implicit} description of the D-MLBN model on $\cd$. 

\begin{ex}
Again let $\cd$ be the DAG pictured on the left in Figure \ref{fig:dmlbnExample}. Then the ideal $I_\cd^{(2)}$ is generated by the polynomials
\begin{align*}
&{p}_{00000}+{p}_{00001}+{p}_{00011}+{p}_{00101}+{p}_{00111}+{p}_{01011}+{p}_{01111}+{p}_{10111}+{p}_{11111}-1, \\
&{p}_{01111}{p}_{10111}-{p}_{00101}{p}_{11111}-{p}_{00111}{p}_{11111}, \\
&{p}_{01011}{p}_{10111}-{p}_{00000}{p}_{11111}-{p}_{00001}{p}_{11111}-{p}_{00011}{p}_{11111}, \\
&{p}_{00111}{p}_{01011}-{p}_{00011}{p}_{01111}, \\
&{p}_{00101}{p}_{01011}-{p}_{00000}{p}_{01111}-{p}_{00001}{p}_{01111}, \\
&{p}_{00011}{p}_{0010 1}-{p}_{00000}{p}_{00111}-{p}_{00001}{p}_{00111}.
\end{align*}
\end{ex}

As we noted before, the state space of the D-MLBN model is a distributive lattice. Hibi showed in \cite{TH87} that there is a toric ideal naturally associated to such a lattice. In \cite{BES07} the authors show that the ideal of the CBN model is  the toric ideal defined by Hibi, after a suitable change of coordinates. So it is natural to attempt to extend this result from the binary case to D-MLBN models with an arbitrary number of states. We describe the construction of Hibi here but for additional information we refer the reader to \cite{HHO18} or \cite{TH87}. Let $L$ be a distributive lattice and recall that there is unique poset $P$ (up to isomorphism) such that $L = J(P)$ (see \cite[Thm 3.4.1]{RS12}).  Let $P$ have ground set $[n]$.
Then the map
\begin{align*}
    \varphi_L :  \rr[q_g : g \in L]  &\to  \rr[t, x_1, \ldots x_n] \\
                    q_g           &\mapsto  t \prod_{i \in g}x_i
\end{align*}
has kernel $I_L = \ker(\phi_L)$ generated by
$I_L = \langle q_g q_h - q_{g \wedge h} q_{g \vee h} : g, h \in \cg \mbox{ are incomparable} \rangle$ \cite{TH87}.  Recall that elements $g$ and $h$ in a poset are \emph{incomparable} if neither $g \leq h$ nor $h \leq g$. The following example illustrates this construction.

\begin{ex}
Let $P$ be the poset pictured on the left in Figure \ref{fig:CBNexample} whose lattice of order ideals $L = J(P)$ is pictured on the right. Then $\varphi_L$ is given by
\begin{align*}
q_{00000} &= t                   &
q_{10000} &= t x_1              \\
q_{01000} &= t x_2               &
q_{11000} &= t x_1 x_2          \\
q_{11100} &= t x_1 x_2 x_3      &
q_{11110} &= t x_1 x_2 x_3 x_4   \\
q_{11101} &= t x_1 x_2 x_3 x_5  &
q_{11111} &= t x_1 x_2 x_3 x_4 x_5. 
\end{align*}
There are two pairs of incomparable elements in the lattice $L$ which are $(1,0,0,0,0)$ and $(0,1,0,0,0)$ as well as $(1,1,1,0)$ and $(1,1,1,0,1)$. This means the generators of the ideal $I_L$ are the binomials
\[
q_{10000}q_{01000}-q_{00000}q_{11000},~~ q_{11110}q_{11101}-q_{11100}q_{11111}
\]
which correspond to these two pairs of incomparable elements. 
\end{ex}

We are now ready to state our main result. 

\begin{thm}
\label{thm:toricDMLBN}
Let $\cd = (V, E)$ be a DAG and $I_\cd^{(k)}$ be vanishing ideal of the $k$-state D-MLBN model on $\cd$. Then after homogeneizing the map $\psi_\cd^{(k)}$ and applying the linear change of coordinates 
\[
q_g = \sum_{h \leq g} p_h
\]
on $\rr[p_g]$ the ideal $I_\cd^{(k)}$ is equal to the toric ideal associated to the distributive lattice $J(D^\mathrm{tr} \times \mathbf{k-1})$ with generating set
\[
I_\cd^{(k)} = \langle q_g q_h - q_{g \wedge h} q_{g \vee h} : g, h \in \cg \mbox{ are incomparable} \rangle.
\]
\end{thm}
\begin{proof}
First we note that for any $g \in \cg = \cg(\cd, k)$ we have that
\begin{equation}
\label{eqn:orderIdealProb}
    \sum_{h \leq g} p_h = \prod_{i \in V} \sum_{\ell \leq g_i} \theta_\ell^{(i)}. 
\end{equation}
This leads us to consider a transform of the parameter space given by
\begin{equation}
\label{eqn:paramTransform}
    \alpha_j^{(i)} = \sum_{\ell \leq j} \theta_\ell^{(i)}.  
\end{equation}
Note that the matrix of this transformation is block diagonal and each block can be made into a lower triangular matrix with ones on the diagonal so this is truly a linear change of coordinates. Combining Equations \ref{eqn:orderIdealProb} and \ref{eqn:paramTransform} we see that in the transformed coordinates, the map $\psi_\cd^{(k)}$ is given by 
\[
q_g = t\prod_{i \in V} \alpha_{g_i}^{(i)}. 
\]
We also have introduced a new variable $t$ which homogenizes the parameterization so that $I_\cd^{(k)}$ will be a homogeneous ideal. This simply removes the trivial relation that all of the coordinates sum to one. 

We now consider the parameterization of the Hibi ideal associated to $I_{L}$ where $L = J(D^\mathrm{tr} \times \{0, 1, \ldots k-2\})$. By \cite[Proposition 3.5.1]{RS12}, there is a bijection between order-preserving maps $g \in \cg$ and order ideals of the poset $D^\mathrm{tr} \times \{0, 1, \ldots, k-2\}$. Under this bijection, a map $g$ is sent to an order ideal $\tilde{g} = \{(i, r) \in D^\mathrm{tr} \times \{0, 1, \ldots, k-2\} :  0 \leq r \leq k - g_i - 1\}$. Then $I_L$ is the kernel of the map
\begin{align*}
    \varphi_L :  \rr[q_g : g \in L]  &\to  \rr[t, x_r^{(i)} : i \in [n], r \in {0, \ldots, k-2}] \\
                    q_g           &\mapsto  t \prod_{(i, r) \in \tilde{g}}x_r^{(i)}.
\end{align*} 
At first, this parameterization might look quite different when compared to the parameterization $\psi_\cd^{(k)}$ but we can transform the parameter space of $\psi_\cd^{(k)}$ again so that they agree. Consider the transform given by 
\[
\alpha_j^{(i)} = \prod_{r = 0}^{k - j - 2} x_{r}^{(i)}
\]
and note that this is invertible with inverse given by
\[
x_{i-j}^{(i)} = \frac{\alpha_{j-1}^{(i)}}{\alpha_j^{(i)}}. 
\]
After applying this transform on the parameter space of $\psi_\cd^{(k)}$ the map is is given by
\[
q_g = t \prod_{i \in v} \prod_{r = 0}^{k - g_i - 2} x_r^{(i)} = t\prod_{(i, r) \in \tilde{g}} x_r^{(i)}
\]
where the last equality follows directly from the definition of $\tilde{g}$. Since the ideals $I_\cd^{(k)}$ and $I_L$ are now the kernel of the exact same map, they are  equal and the generating set stated above is exactly the generating set described by Hibi in \cite{TH87}. 
\end{proof}

\begin{rmk}
Note that the coordinate transform that takes the $p_g$ coordinates to the $q_g$ coordinates corresponds to M\"obius inversion on the lattice $\cg$. This is the same transform that is used in \cite{BES07} for CBN models but in this case the resulting toric ideal is that associated to $J(D^\mathrm{tr})$ since it is a 2-state model. 
\end{rmk}

We conclude with two examples which illustrate the previous theorem. 

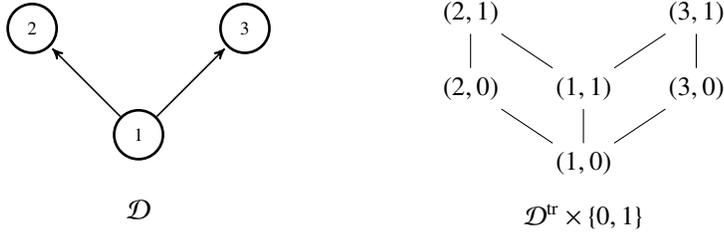
\begin{figure}
    \centering
    \begin{subfigure}[b]{0.35 \linewidth}
        \begin{tikzpicture}[->, >=stealth', shorten >=1pt, auto, node distance=2cm, semithick]
        \tikzstyle{every node}=[circle, line width =1pt, font=\scriptsize, minimum height =0.65cm]
            \node (i1) [draw] {1};
            \node (i2) [above left of = i1, draw] {2};
            \node (i3) [above right of  = i1, draw] {3};
            
            \path (i1) edge (i2);
            \path (i1) edge (i3);
            
            \node (label) [below of = i1, yshift = 1cm] {\scalebox{1.5}{$\cd$}};
        \end{tikzpicture}
    \end{subfigure}
    \begin{subfigure}[b]{0.35 \linewidth}
        \begin{tikzpicture}
            \node (i10) {$(1,0)$};
            \node (i11) [above of = i10]{$(1,1)$};
            \node (i20) [above of = i10, xshift =-1.5cm] {$(2,0)$};
            \node (i30) [above of = i10, xshift = 1.5cm] {$(3,0)$};
            \node (i21) [above of = i20] {$(2,1)$};
            \node (i31) [above of = i30] {$(3,1)$};
        
            \path (i10) edge (i11);
            \path (i10) edge (i20);
            \path (i10) edge (i30);
            \path (i11) edge (i21);
            \path (i11) edge (i31);
            \path (i20) edge (i21);
            \path (i30) edge (i31);

            \node (label) [below of = i0, yshift = .3cm] {\scalebox{1}{$\cd^\mathrm{tr} \times \{0,1\}$}};
        \end{tikzpicture}
    \end{subfigure}
    \caption{A DAG $\cd$ and the poset $\cd^\mathrm{tr} \times \{0,1\}$ whose order ideals correspond to the order-preserving maps from $\cd$ to $\{0, 1, 2\}$} 
    \label{fig:threeVertGraph}
\end{figure}

\begin{ex}
Let $\cd$ be the graph pictured in Figure \ref{fig:threeVertGraph} on the left. Consider the state $g = (0, 1, 2) \in \cg$ of the 3-state D-MLBN model on $\cd$. The original probability of this state under the model is $p_{012} = \theta_0^{(1)}\theta_1^{(2)}\theta_2^{(3)}$. Then after our first coordinate transform
\[
q_{012} = \sum_{h \leq g} p_h = p_{000}+p_{001}+p_{002}+p_{010}+p_{011}+p_{012} = \theta_0^{(1)}(\theta_0^{(2)}+\theta_1^{(2)})(\theta_0^{(3)}+\theta_1^{(3)} + \theta_2^{(3)}) = \alpha_0^{(1)}\alpha_1^{(2)}.
\]
Note that we omit $\alpha_2^{(3)}$ since the parameter corresponding to the state $k-1$ is always 1. Under our second transform of the parameter space the relevant parameters become $\alpha_0^{(1)} = x_0^{(1)}x_1^{(1)}$ and $\alpha_1^{(2)} = x_0^{(2)}$. After homogenizing with a new parameter $t$ and applying this second transform we have that
\[
\psi_\cd^{(3)}(q_{012}) = t x_0^{(1)} x_1^{(1)} x_0^{(2)}.
\]

The state $g$ also corresponds to the order ideal $\tilde{g} = \{(1,0), (1,1), (2,0)\} \in L = J(\cd^\mathrm{tr} \times \{0,1\})$. This means the map $\varphi_L$ takes $q_g$ to
\[
\varphi_L(q_{012}) = t x_0^{(1)}x_1^{(1)}x_0^{(2)}
\]
and we can see that $\varphi_L(q_g) = \psi_\cd^{(3)}(q_g)$ as was shown in the proof of Theorem \ref{thm:toricDMLBN}. 
\end{ex}

\begin{ex}
Let $\cd$ be the graph pictured in Figure \ref{fig:dmlbnExample}. Then after homogenizing the parameterization $\psi_\cd^{(2)}$ and applying the coordinate transform described in Theorem \ref{thm:toricDMLBN} the ideal $I_\cd^{(2)}$  is generated by the polynomials
\begin{align*}
q_{0 1 1 1 1}q_{1 0 1 1 1}-q_{0 0 1 1 1}q_{1 1 1 1 1}, \\
q_{0 1 0 1 1}q_{1 0 1 1 1}-q_{0 0 0 1 1}q_{1 1 1 1 1}, \\
q_{0 0 1 1 1}q_{0 1 0 1 1}-q_{0 0 0 1 1}q_{0 1 1 1 1}, \\
q_{0 0 1 0 1}q_{0 1 0 1 1}-q_{0 0 0 0 1}q_{0 1 1 1 1}, \\
q_{0 0 0 1 1}q_{0 0 1 0 1}-q_{0 0 0 0 1}q_{0 0 1 1 1}.
\end{align*}
Note that the first monomial in each polynomial corresponds to a pair of incomparable elements $g, h \in \cg$ while the second corresponds to their meet and join which are given by taking coordinate-wise minimums and maximums respectively. 
\end{ex}

\section*{Acknowledgments}
Benjamin Hollering and Seth Sullivant were partially supported by the US National Science Foundation
(DMS 1615660).

\bibliography{ref.bib}{}
\bibliographystyle{plain}
\end{document}